\documentclass[12pt]{amsart}

\usepackage{amssymb, epsfig} 
\usepackage{amsmath} 
\usepackage{latexsym} 

\numberwithin{equation}{section}  
  
\renewcommand{\a}{\alpha}

\newcommand{\e}{\varepsilon}

\renewcommand{\div}{{\rm div}}

\newcommand{\R}{{\mathbb R}}  

\newcommand{\1}{{1\hspace{-1.4 mm}1}}

\newcommand{\Lc}{{\mathcal L}}  

\newcommand{\Slc}{{\mathcal S}}

\newtheorem{theorem}{Theorem}[section]  
  
\newtheorem{lemma}[theorem]{Lemma}

\newtheorem{definition}[theorem]{Definition}

\theoremstyle{definition}

\theoremstyle{remark}  
\newtheorem{remark}[theorem]{Remark}

\title[Existence results for nonlocal equations]
{Existence of weak solutions for general nonlocal and nonlinear second-order parabolic equations}  
\author[G. Barles, P. Cardaliaguet, O. Ley and A. Monteillet]
{Guy Barles \and Pierre Cardaliaguet
 \and Olivier Ley \and Aur\'elien Monteillet}

\address{ (G. Barles, O. Ley)  
Laboratoire de Math\'ematiques et Physique Th\'eo\-ri\-que \\
F\'ed\'eration Denis Poisson \\
Universit\'e de Tours \\
Parc de Grandmont, 37200 Tours, France \\ {\tt \{barles,ley\}@lmpt.univ-tours.fr}}
\address{
(P. Cardaliaguet, A. Monteillet) Laboratoire de Math\'ematiques
  \\ CNRS UMR 6205\\  
Universit\'e de Brest \\ 6 Av. Le Gorgeu
BP 809, 29285 Brest, France \\
 {\tt \{pierre.car\-da\-liaguet, aurelien.monteillet\}@univ-brest.fr}}

\thanks{This work was partially supported by the ANR (Agence Nationale de la Recherche) 
through MICA project (ANR-06-BLAN-0082)}
 
 
\begin{document}  
  
\maketitle

\begin{abstract} 
In this article, we provide existence results for a general class 
of nonlocal and nonlinear second-order parabolic equations. 
The main motivation comes from front propagation theory in the 
cases when the normal velocity depends on the 
moving front in a nonlocal way. Among applications, we present level-set equations 
appearing in dislocations' theory and in the study of Fitzhugh-Nagumo systems.
\end{abstract} 
%%%%%%%%%%%%%%%%%%%%%%%%%%%%%%%%%%%%%%%%%%%%%%%%%%%%%%%%%%%%%%%%%%%%%%%%%%%%%
\section{Introduction}  

We are concerned with a class of nonlocal and nonlinear parabolic 
equations which can be written as
\begin{eqnarray} \label{eq-gene}
\left\{
\begin{array}{ll}
 u_t = H[\1_{\{ u\geq 0 \}}](x,t,u,Du,D^2 u)& {\rm in}\
 \R^N\times (0,T), \\
 u(\cdot , 0)= u_0 & {\rm in}\ \R^N,
\end{array}
\right.
\end{eqnarray}
where $u_t$, $Du$ and $D^2u$ stand respectively for the time
derivative, 
gradient and Hessian matrix with respect to the space variable 
$x$ of $u:\R^N \times [0,T]\to \R$ and where $\1_A$ denotes the
indicator 
function of a set $A$. The initial datum $u_0$ is a bounded 
and Lipschitz continuous function on $\R^N$.
\vspace*{.2cm}

For any indicator function $\chi:\R^N \times [0,T]\to \R$, 
or more generally for any 
$\chi \in L^{\infty}(\R^N \times [0,T];[0,1])$, $H[\chi]$ 
denotes a function of $(x,t,r,p,A)\in \R^N \times [0,T] 
\times \R\times \R^N \setminus \{0\} \times \Slc_N$, 
where $\Slc_N$ is the set of real, $N\times N$ symmetric 
matrices. For almost any $t\in [0,T]$, $(x,r,p,A) \mapsto 
H[\chi](x,t,r,p,A)$ is a continuous function on $\R^N 
\times \R\times \R^N \setminus \{0\} \times \Slc_N$ with 
a possible singularity at $p=0$ (when considering 
geometrical equations, see for instance Giga \cite{giga06}), 
while $t \mapsto H[\chi](x,t,r,p,A)$ is a bounded
measurable function for all $(x,r,p,A)\in \R^N 
\times \R\times \R^N \setminus \{0\} \times \Slc_N.$ 
We recall that the equation is said to be degenerate elliptic 
(or here parabolic) if, for any 
$\chi \in L^{\infty}(\R^N \times [0,T];[0,1])$, 
for any $(x,r,p)\in \R^N \times \R
\times \R^N \setminus \{0\}$, for almost every 
$t \in [0,T]$ and for all $A,B \in \Slc_N$, one has
$$
H[\chi](x,t,r,p,A) \leq H[\chi](x,t,r,p,B)\quad\hbox{if  }A\leq B,
$$
where $\leq$ stands for the usual partial ordering for symmetric matrices.
\vspace*{.2cm}

Such equations arise typically when one aims at describing, through the 
``level-set approach'', the motion of a family $\{K(t)\}_{t \in [0,T]}$ 
of closed subsets of $\R^N$ evolving with a nonlocal velocity. Indeed,
following the main idea of the level-set approach, it is 
natural to introduce a function $u$ such that
$$
K(t)=\{x \in \R^N;\; u(x,t)\geq 0\}\; ,
$$
and \eqref{eq-gene} can be seen as the level-set equation for $u$. 
In this framework, the nonlinearity $H$ corresponds to the velocity 
and, in the applications we have in mind, it depends not only on the time, the position of the front, the normal direction and the curvature 
tensor but also on nonlocal properties of $K(t)$ which are 
carried by the dependence in $\1_{\{ u\geq 0 \}}$. We may face 
rather different nonlocal dependences and this is why we have 
chosen this formulation: in any case, the equation appears 
as a well-posed equation if we would consider the nonlocal 
dependence (i.e. $\1_{\{ u\geq 0 \}}$) as being fixed; in other words,
the $H[\chi]$-equation enjoys ``good'' properties. 
\vspace*{.2cm}

Finally, we recall that, still in the case of level-set equations, 
the function $u_0$ is used to represent the initial front, \textit{i.e.}
\begin{equation} \label{cond-initiale}
\{u_0 \geq 0\}=K_0 \quad \text{and} \quad \{u_0 = 0\}=\partial{K_0}   
\end{equation}  
for some fixed compact set $K_0 \subset \R^N$.
We refer the reader to \cite{giga06} and the references therein
for precisions.
\vspace*{.2cm}

Now we turn to the main examples we have in mind.
\vspace*{.2cm}

%The main issue linked to Equation \eqref{eq-gene} is that it bears 
%a \textit{nonlocal} dependence, \textit{i.e.} a dependence 
%in $\1_{\{u\geq 0\}}$. We point out that we place ourselves in a 
%general setting where $H[\1_{\{ u\geq 0 \}}](x,t,Du,D^2u)$ at time 
%$t$ depends on the values of $\1_{\{ u(\cdot,s) \geq 0 \}}$ for all 
%past times $s\leq t$ (see the Fitzhugh-Nagumo system below), but 
%even possibly for future times $s>t$. Typical examples of such 
%nonlocal equations, that we will study in details below, are
%\vspace{.2cm}

\noindent $1.$ Dislocation dynamics equations 
$$
\begin{aligned}
&u_t = (c_0 (\cdot ,t) \star \1_{\{u(\cdot,t)\geq 0\}}(x) +c_1(x,t))|D u|,\\
\text{or} \quad & u_t=\left[\div \left(\frac{Du}{|Du|}\right) +
c_0(\cdot,t) \star \1_{\{ u(\cdot,t)\geq 0 \}}(x) +
c_1(x,t)\right]|Du|,
\end{aligned}
$$ 
where $$c_0 (\cdot ,t) \star \1_{\{u(\cdot,t)\geq 0\}}(x)=\int_{\R^N} c_0(x-y,t)  
\1_{\{u(\cdot,t)\geq 0\}}(y)dy$$
and 
$\displaystyle \div \! \left( \! \frac{Du}{|Du|} \! \right) \! (x,t)$ is the mean
curvature of the set $\{u(\cdot,t)= u(x,t)\}$ at~$x.$

Typically, the reasonnable assumptions in this context (see, for example, \cite{bclm07}) 
are the following: $c_0,c_1$ are bounded, continuous functions which are 
Lipschitz continuous in $x$ (uniformly with respect to $t$) and 
$c_0,$ $D_x c_0 \in L^\infty([0,T];L^1(\R^N))$. In particular, 
and this is a key difference with the second example below, 
$c_0$ is bounded.
\vspace{.2cm}

\noindent $2.$ Fitzhugh-Nagumo type systems, which, in a simplified
form, reduce to the nonlocal equation
$$
u_t=\a \left(\int_0^t \int_{\R^N} G(x-y,t-s)  \, \1_{\{u(y,s)\geq 0\}}
\, dy ds \right) |Du|,
$$
where $\alpha:\R\to\R$ is Lipschitz continuous and 
$G$ is the Green function of the heat equation (see \eqref{green}).

There are two key differences with the first example: 
the convolution kernel acts in space and time, and $G$ is not bounded. 
This difference plays a central role when one tries to 
prove uniqueness (cf. \cite{bclm07a}).
\vspace*{.2cm}

\noindent $3.$ Equations of the form 
$$
\begin{aligned}
&u_t=(k-\Lc^N\left(\{u(\cdot,t)\geq 0\} \right))|Du| \qquad (k\in \R),\\
\text{or} \quad & u_t=\left[\div \left(\frac{Du}{|Du|}\right) 
-\Lc^N(\{u(\cdot,t)\geq 0\})\right]|Du|,
\end{aligned} 
$$
where $\Lc^N$ denotes the $N$-dimensional Lebesgue measure and
therefore the velocity of the front at time $t$ depends on the 
volume of $K(t)=\{u(\cdot,t)\geq 0\}.$\\

In the classical cases, the equations of the level-set approach are solved by using 
the theory of viscosity solutions. Nevertheless there are two key
features which may prevent a direct use of viscosity solutions' 
theory to treat the above examples: the main problem is that these 
examples do not satisfy the right monotonicity property. 
This can be seen either through the fact that 
$\{u\geq 0\}\subset \{v\geq 0\}$ does not imply that  
$H[\1_{\{u\geq 0\}}]\leq H[\1_{\{v\geq 0\}}]$, or by remarking that 
the associated front propagations do not satisfy the 
``inclusion principle'' (geometrical monotonicity). 
Indeed, in the dislocation dynamics case, the kernel $c_0$ 
changes sign, which implies the two above facts. Therefore the 
classical comparison arguments of viscosity solutions' 
theory fail, and since existence is also based on these 
arguments through the Perron's method, the existence of 
viscosity solutions to these equations becomes an issue too.
\vspace*{.2cm}

The second (and less important) feature which prevents a 
direct use of the standard level-set approach arguments 
is the form of the nonlocal dependence in $\1_{\{u\geq 0\}}$: 
as shown in Slepcev \cite{slepcev03} and used in the 
present framework (but in the monotone case) in 
\cite{bclm07}, a dependence in $\1_{\{u\geq u(x,t)\}}$ 
is the most adapted to the level-set approach since 
all the level sets of the solutions are treated 
similarly instead of having the $0$-level set 
playing a particular role.
\vspace*{.2cm}

As a consequence, we are going to use a notion of weak 
solutions for \eqref{eq-gene} introduced in \cite{bclm07} (see
Definition \ref{def-weak}), and prove two general existence 
results. As an 
simple application of this theorem, we recover existence 
results for dislocation equations and the Fitzhugh-Nagumo 
system obtained by Giga, Goto and Ishii \cite{ggi92},
Soravia and Souganidis \cite{ss96} and in \cite{bclm07}.
Let us mention that the technique of proof of our results, using Kakutani's fixed point theorem, is the same that is used in \cite{ggi92}. Here we generalize its range of application and combine it with a new stability result of Barles \cite{barles06}. In \cite{bclm07a}, 
we prove the uniqueness of such weak
solutions for these two model equations. Note that the 
issue of uniqueness is a difficult problem and, in general, 
uniqueness does not hold as shown by the counterexample 
developed in \cite{bclm07}.
\vspace*{.2cm}

Another issue of these nonlocal equations is connected to the 
behavior of $H$ with respect to the size of the set $\{u\geq 0\}$. 
Indeed, in the dislocation dynamics case, if $c_0\in
L^{\infty}([0,T];L^1(\R^N))$, then $H[\1_{\{u\geq 0\} }]$ is defined 
without restriction on the size of $\{u\geq 0\}$. The situation is 
the same for the Fitzhugh-Nagumo system. However, if $c_0$ is only 
bounded and not in $L^1$,  or in volume-dependent equations, then the support of 
$\{u\geq 0\}$ has to be bounded for $H[\1_{\{u\geq 0\}}]$ to be
defined. This leads us to distinguish two cases, that we will 
call respectively the unbounded and the bounded case.
\vspace*{.2cm}

The paper is organized as follows: in Section 2 we give a general 
definition of a weak solution to \eqref{eq-gene}. In Section 3 we 
prove existence of such solutions in the unbounded case, and apply 
our result to dislocation equations and the Fitzhugh-Nagumo system. 
In Section 4 we treat the bounded case and give as an application 
an existence result for volume-dependent equations.
\vspace*{.2cm}

\noindent\textbf{Notation:}
In the sequel, $|\cdot|$ denotes the standard euclidean norm in 
$\R^N$, $B(x,R)$ (resp. $\bar{B}(x,R)$) is the open (resp. closed) 
ball of radius $R$ centered at $x \in \R^N$. The notation
$\mathcal{S}_N$ denotes the space of real $N \times N$ symmetric matrices.
  
%%%%%%%%%%%%%%%%%%%%%%%%%%%%%%%%%%%%%%%%%%%%%%%%%%%%%%%%%%%%%%%%%%%%%%%%%%%%%%%%
\section{Definition of weak solutions}

We will use the following definition of weak solutions introduced in \cite{bclm07}. To do so, we use the notion of viscosity solutions for equations with a measurable dependence in time which we call below ``$L^1$-viscosity solution''. We refer the reader to \cite[Appendix]{bclm07} for the definition of $L^1$-viscosity solutions and \cite{ishii85, nunziante90, nunziante92, bourgoing04a,bourgoing04b} for a complete presentation of the theory.

%%%%%
\begin{definition}  \label{def-weak}
Let $u: \R^N \times [0,T] \to \R$ be a continuous function. We say that $u$ is 
a weak solution of \eqref{eq-gene} if there exists  $\chi \in L^{\infty}(\R^N \times [0,T];[0,1])$ such that
\begin{enumerate}  
\item $u$ is a $L^1$-viscosity solution of  
\begin{eqnarray} \label{eq-gene2}
\left\{
\begin{array}{ll}
 u_t(x,t)= H[\chi](x,t,u,Du,D^2 u)& {in}\
 \R^N\times (0,T), \\
 u(\cdot , 0)= u_0 & {in}\ \R^N.
\end{array}
\right.
\end{eqnarray}
\item For almost all $t \in [0,T]$, 
\begin{equation}\label{eq-gene22}
\1_{\{ u(\cdot,t) > 0 \}} 
\leq \chi(\cdot,t) \leq \1_{\{ u(\cdot,t) \geq 0 \}}
\ \ \ {a.e. \  in} \ \R^N.
\end{equation}
\end{enumerate}  
Moreover, we say that $u$ is a classical solution of \eqref{eq-gene} 
if in addition, for almost all $t \in [0,T]$,  
$$  
\1_{\{u(\cdot,t)>0\}}=\1_{\{u(\cdot,t)\geq0\}}\ \ \ {a.e. \  in} \ \R^N. 
$$  
\end{definition}

\begin{remark} If for any fixed $\chi\in L^{\infty}(\R^N \times [0,T];[0,1])$ the map $H[\chi]$ is geometric, then the map $\chi$ defined by
(\ref{eq-gene2})-(\ref{eq-gene22}) only depends on the $0-$level-set of the initial condition $u_0$, 
as in the classical level-set approach. Indeed, let $u_0^1:\R^N\to \R$ be another bounded continuous map 
such that
$$
\{u_0\geq 0\}=\{u_0^1\geq 0\}\; {\rm and }\; 
\{u_0> 0\}=\{u_0^1> 0\}\;,
$$
and $u^1$ be the solution to (\ref{eq-gene2}) with the same $\chi$ but with initial condition $u_0^1$
(under the assumptions of Theorem \ref{existence-gene} or Theorem \ref{existence-gene2} such a solution exists and is unique). 
Then from the key property of geometric equations (see for instance \cite{giga06}) we will have, for almost all $t\in [0,T]$,
$$
\1_{\{ u(\cdot,t) > 0 \}} =\1_{\{ u^1(\cdot,t) > 0 \}} 
\leq \chi(\cdot,t) \leq \1_{\{ u(\cdot,t) \geq 0 \}}= \1_{\{ u^1(\cdot,t) \geq 0 \}}
\ \ \ {\rm a.e. \  in} \ \R^N
$$
This means that the map $\chi$ can be interpreted as the weak solution of a nonlocal geometric flow.
\end{remark}

%%%%%%%%%%%%%%%%%%%%%%%%%%%%%%%%%%%%%%%%%%%%%%%%%%%%%%%%%%%%%%%%%%%%%%%%%
%%%%%%%%%%%%%%%%%%%%%%%%%%%%%%%%%%%%%%%%%%%%%%%%%%%%%%%%%%%%%%%%%%%%%%%%%
\section{Existence of weak solutions to \eqref{eq-gene} (unbounded case)}  
\label{sec:existence}

In this section, we are interested in the case where the Hamiltonian
$H[\chi]$ 
is defined without any restriction on the size 
of the support of $\chi$.

%%%%%%%%%%%%%%%%%%%%%%%%%%%%%%%%%%%%%%%%%%%%%%%%%%%%%%%%%%%%%%%%%%%%%%%%%
\subsection{The existence theorem}

%%%%

We first state some assumptions which we use here but also in the next
sections. 
To avoid repeating them, we are going to formulate assumptions on the 
nonlinearities $H[\chi]$ which have to be satisfied for any $\chi \in
X$ 
(and uniformly for such $\chi$) where $X$ is a subset of 
$L^{\infty}(\R^N \times [0,T];[0,1])$. We use a different $X$ 
in this section and for the ``bounded'' case.
\vspace*{.2cm}

\noindent{\bf (H1-$X$)} $(i)$ For any $\chi \in X$, Equation
\eqref{eq-gene2} 
has a bounded uniformly continuous $L^1$-viscosity solution
$u$. Moreover, 
there exists a constant $L>0$ independent 
of $\chi \in X$ such that $|u|_{\infty} \leq L.$

\vspace*{.2cm}

$(ii)$ For any fixed $\chi \in X$, a comparison principle holds for 
Equation \eqref{eq-gene2}: if $u$ is a bounded, upper-semicontinuous $L^1$-viscosity 
subsolution of \eqref{eq-gene2} in $\R^N \times (0,T)$ and $v$ is a bounded, lower-semicontinuous $L^1$-viscosity supersolution of \eqref{eq-gene2} 
in $\R^N\times (0,T)$
with $u(\cdot ,0) \leq v(\cdot ,0)$ in $\R^N$, then
$u\leq v$ in $\R^N \times [0,T]$. 

In the same manner, if $u$ is a bounded, upper-semicontinuous $L^1$-viscosity 
subsolution of \eqref{eq-gene2} in $B(x,R) \times (0,T)$ for some $x \in \R^N$ and $R>0$,
and $v$ is a bounded, lower-semicontinuous $L^1$-viscosity supersolution of \eqref{eq-gene2} 
in $B(x,R)\times (0,T)$
with $u(y,t) \leq v(y,t)$ if $t=0$ or  $|y-x|=R$, then
$u\leq v$ in $B(x,R)\times[0,T]$.
\vspace*{.2cm}

\noindent{\bf (H2-$X$)} $(i)$ For any compact subset 
$K\subset \R^N \times \R \times \R^N \setminus \{0\}\times\mathcal{S}_N$, 
there exists a (locally bounded) modulus of continuity 
$m_K :[0,T]\times \R^+ \to \R^+$ such 
that $m_K (\cdot,\e)\to 0$ in $L^1(0,T)$ as $\e \to 0$, and
$$
| H[\chi](x_1,t,r_1,p_1,A_1) - H[\chi](x_2,t,r_2,p_2,A_2) | 
\leq \qquad \qquad \qquad \qquad \qquad \qquad \qquad \ $$
$$\ \qquad \qquad \qquad \qquad \qquad m_K (t,|x_1-x_2|+|r_1-r_2| + |p_1-p_2|+|A_1-A_2|)
$$
for any $\chi \in X$, for almost all $t\in [0,T]$ and all $(x_1,r_1,p_1,A_1),$ $(x_2,r_2,p_2,A_2)\in K$.

\vspace*{.2cm}

$(ii)$ There exists a bounded function $f(x,t,r)$, which is continuous in $x$ 
and $r$ for almost every $t$ and mesurable in $t$, such that: for any neighborhood $V$ of $(0,0)$ in 
$\R^N \setminus \{0\}  \times \Slc_N$ and any compact subset $K\subset\R^N \times \R$, 
there exists a modulus of continuity $m_{K,V} :[0,T]\times \R^+ \to \R^+$
such that $m_{K,V}(\cdot,\e)\to 0$ in $L^1(0,T)$ as $\e \to 0$, and
$$
| H[\chi](x,t,r,p,A) -f(x,t,r) | \leq m_{K,V}(t,|p|+|A|)
$$ 
for any $\chi \in X$, for almost all $t\in [0,T]$, all $(x,r) \in K$ and $(p,A)\in V$.

\vspace*{.2cm}

$(iii)$ If $\chi_n  \rightharpoonup \chi$ weakly-$*$ in 
$L^{\infty}(\R^N \times [0,T];[0,1])$ with $\chi_n,\chi\in X$ for all $n$, then for all 
$(x,t,r,p,A)\in \R^N\times [0,T] \times \R\times  \R^N \setminus 
\{0\} \times \mathcal{S}_N,$
\begin{eqnarray*}
\int_0^t H[\chi_n](x,s,r,p,A)ds 
\underset{n \to +\infty}{\longrightarrow} \int_0^t H[\chi](x,s,r,p,A)ds
\end{eqnarray*}
locally uniformly  for $t \in [0,T].$ \\

We finally add an assumption which is not optimal on the behavior of
$H[\chi]$ with respect to $r$ but to which we can reduce, in most cases, 
after some change of unknown functions like $u \to u\exp(\gamma t)$:
\vspace*{.2cm}

\noindent{\bf (H3-$X$)} For any $\chi \in X$, for almost every $t\in [0,T]$, for all $(x,p,A)\in \R^N\times \R^N \setminus 
\{0\} \times \mathcal{S}_N,$ and for any $r_1\geq r_2$
$$H[\chi](x,t,r_1,p,A) \geq H[\chi](x,t,r_2,p,A)\; .$$
Of course, gathering {\bf (H2-$X$)} and {\bf (H3-$X$)}, 
it is easy to show that $f$ satisfies the same property.

\vspace{.2cm}

We have chosen to state Assumption {\bf (H1-$X$)} in this form which may look
artificial: it means that we have existence, uniqueness of a 
continuous $L^1$-viscosity solution $u$ associated to any measurable {\em fixed} function 
$0\leq \chi \leq 1.$ For conditions on $H$ under which {\bf (H1-$X$)} is verified, we refer to 
\cite{bourgoing04a,nunziante92} and Section \ref{sec:appli-unbound}.
Moreover, {\bf (H1-$X$)} $(i)$ states that the $u$'s are bounded uniformly
with respect to $\chi\in X $: for the 
geometrical equations of the level-set approach, this uniform bound on $u$ 
is automatically satisfied with
$L=||u_0||_\infty$ if {\bf (H1-$X$)} $(ii)$ holds, using that, in this case, constants 
are $L^1$-viscosity solutions of \eqref{eq-gene2}. 
Assumption {\bf (H2-$X$)} comes from 
\cite{barles06} and will be used to apply a stability result 
for equations with $L^1$-dependence in time.

Our general existence theorem is the following:
%%%%%%%%%%%%%%
\begin{theorem}  \label{existence-gene}
Assume that {\bf (H1-$X$)}, {\bf (H2-$X$)} and {\bf (H3-$X$)} hold with $X= L^{\infty}(\R^N \times [0,T];[0,1])$.
Then there exists at least a weak solution to \eqref{eq-gene}. 
\end{theorem}  
%%%%%%%%%%%%%%
\begin{remark} See also \cite{barles06} for the stability of weak solutions. 
\end{remark}
%%%%%%%%%%%%%%
\begin{proof} From {\bf (H1-$X$)}, the set-valued mapping 
\begin{eqnarray*}
\begin{array}{cccl}  
\xi : & X & \rightrightarrows &  X\\
      &  \chi                            &\mapsto            &  \big\{
      \chi'\; : \; \1_{\{ u(\cdot,t) > 0 \}}  \leq \chi'(\cdot,t) \leq
      \1_{\{ u(\cdot,t) \geq 0 \}} \; \text{for
        almost all} \;  t \in [0,T] ,\\[1mm] 
      &                                  &                   & 
      \quad \text{where} \; u \; 
      \text{is the $L^1$-viscosity solution of} \; \eqref{eq-gene2} \big\},  
\end{array}
\end{eqnarray*}  
is well-defined. Clearly, there exists a weak solution to
\eqref{eq-gene} 
if there exists a fixed point $\chi$ of
$\xi$, which means that $\chi \in \xi(\chi)$. 
In this case the corresponding $u$ is a weak solution of \eqref{eq-gene}. 
We therefore aim at using Kakutani's fixed point theorem for set-valued
mappings (see \cite[Theorem 3 p. 232]{ac84}). 
  
In the Hausdorff convex space $L^\infty(\R^N \times [0,T];\R)$, the subset $X$ is convex and compact for
the $L^\infty$-weak-$*$ topology (since it is closed and bounded). In the same way,
for any $\chi \in X$, $\xi(\chi)$ is
a non-empty convex compact subset of $X$ for
the $L^\infty$-weak-$*$ topology.

Let us check that $\xi$ is upper semicontinuous for this topology. It suffices
to show that, if  
\begin{eqnarray*}
\chi_n \in X \mathop{\rightharpoonup}_{L^\infty\text{-weak-}*} \chi  
\quad \text{and} \quad
\chi'_n \in \xi(\chi_n) \mathop{\rightharpoonup}_{L^\infty\text{-weak-}*} \chi',
\end{eqnarray*}   
then  
\begin{eqnarray*}
\chi' \in \xi(\chi).  
\end{eqnarray*} 
Let $u_n$ be the unique $L^1$-viscosity solution of \eqref{eq-gene2} 
associated to $\chi_n$ by
{\bf (H1-$X$)}. Using {\bf (H1-$X$)}, we know that the $u_n$'s 
are uniformly bounded. We can therefore define the half-relaxed 
limits\footnote{$\displaystyle {\rm lim\,sup}^* u_n (x,t)
:=\limsup_{{\displaystyle{\mathop{\scriptstyle{(y,s)\to(x,t)}}_{n\to \infty}}}}
\,u_n(y,s)$ and $\displaystyle {\rm lim\,inf}_* u_n (x,t)
:=\liminf_{{\displaystyle{\mathop{\scriptstyle{(y,s)\to(x,t)}}_{n\to \infty}}}}\,u_n(y,s).$}
\begin{eqnarray*}
\overline{u}={\rm lim\,sup}^* u_n \quad
 {\rm and  } \ \underline{u}={\rm lim\,inf}_* u_n. {\rm }
\end{eqnarray*} 
From {\bf (H2-$X$)} (convergence of the Hamiltonians), we can apply Barles' stability result
\cite[Theorem 1.1]{barles06} to obtain that $\overline{u}$ 
(respectively $\underline{u}$) is a  $L^1$-viscosity subsolution 
(respectively supersolution) of \eqref{eq-gene2} associated to $\chi$. 

\vspace{.2cm}

In order to apply {\bf (H1-$X$)} $(ii)$ (comparison), 
we first have to show that
$$
\overline{u}(x,0)\leq u_0(x) \leq \underline{u}(x,0)\quad\hbox{in }\R^N\; .
$$
To do so, we examine {\bf (H2-$X$)} and deduce that $H[\chi]
(x,t,0,p,A)$ is 
bounded if $p$ and $A$ are bounded, uniformly with respect to
$\chi \in X$. 
Moreover $u_0$ is Lipschitz continuous, and therefore for all 
$0<\e \leq 1$, we have, for any $x,y \in \R^N$,
$$ 
u_0(y) \leq u_0(x) +||Du_0||_\infty |x-y| \leq u_0(x) 
+ \frac{|x-y|^2}{2\e^2} + \frac{||Du_0||_\infty\e^2}{2}\; .
$$
We fix $x$ and we argue in the ball $B(x,\e)$. Using  {\bf (H3-$X$)} and the fact that the $H[\chi_n]$'s 
are locally bounded, the function
$$ 
\psi_\e(y,t) = u_0(x) + \frac{|x-y|^2}{2\e^2} 
+ \frac{||Du_0||_\infty\e^2}{2} + C_\e t
$$
is a supersolution of the $H[\chi_n]$-equation in the ball $B(x,\e)$ provided 
that $C_\e$ is  large enough.  By {\bf (H1-$X$)} $(ii)$ (comparison), we obtain
$$ u_n(y,t) \leq  \psi_\e(y,t) 
\quad\hbox{in  }B(x,\e) \times [0,T]\; ,$$
and then
$$ 
\overline{u}(y,t) \leq  \psi_\e(y,t) 
\quad\hbox{in  }B(x,\e) \times [0,T]\; .
$$
Examining the right-hand side at $(y,t)=(x,0)$ and letting $\e\to 0$ provides the 
inequality $\overline{u}(x,0) \leq u_0 (x)$. An 
analogous argument gives $\underline{u}(x,0) \geq u_0 (x)$.

By {\bf (H1-$X$)} $(ii)$ (comparison), we therefore 
have $\overline{u}\leq \underline{u}$ in $\R^N$,
which implies that $u:= \overline{u}=\underline{u}$ is 
the unique  continuous $L^1$-viscosity solution of 
\eqref{eq-gene2} associated to $\chi,$ as well as the 
local uniform convergence of $u_n$ to $u$.

\vspace{.2cm}

Moreover, since $\chi'_n\in \xi(\chi_n),$ we have,
for all $\varphi\in L^1(\R^N\times [0,T];\R_+),$
\begin{eqnarray*} 
\int_0^T \int_{\R^N} \varphi\,\1_{\{ u_n(\cdot,t) > 0 \}} 
\leq \int_0^T \int_{\R^N} \varphi\,\chi'_n
 \leq \int_0^T \int_{\R^N} \varphi\,\1_{\{ u_n(\cdot,t) \geq 0 \}}.
\end{eqnarray*}
Since $\displaystyle \chi'_n  \mathop{\rightharpoonup}_{L^\infty\text{-weak-}*} \chi',$
applying Fatou's Lemma, we get
\begin{eqnarray*}
\int_0^T \int_{\R^N} \varphi \,{\rm lim\,inf} \,\1_{\{ u_n(\cdot,t) > 0  \}} 
\leq \int_0^T \int_{\R^N} \varphi\,\chi'
 \leq \int_0^T \int_{\R^N} \varphi \,{\rm lim\,sup} \,\1_{\{ u_n(\cdot,t) \geq 0 \}}.
\end{eqnarray*}
But $\1_{\{ u(\cdot,t)>0 \}} \leq {\rm lim\,inf} \,\1_{\{ u_n(\cdot,t)>0 \}}$
and ${\rm lim\,sup} \,\1_{\{ u_n(\cdot,t) \geq 0 \}}\leq 
\1_{\{ u(\cdot,t) \geq 0 \}}.$ It follows that
\begin{eqnarray*}
\1_{\{ {u}(\cdot,t)>0 \}}\leq \chi'\leq \1_{\{ {u}(\cdot,t) \geq 0 \}}
\quad \text{for a.e.} \, t\in [0,T],
\end{eqnarray*}
and therefore $\chi'\in \xi (\chi).$

We infer the existence of a weak solution of \eqref{eq-gene} by  
Kakutani's fixed point theorem \cite[Theorem 3 p. 232]{ac84},
as announced.
\end{proof}
%
%When applying the previous theorem to geometrical nonlocal equations,
%we will sometimes need the following 
%assumption on the initial datum $u_0$ (see \cite{ley01, bl06} for details):\\

%%%
%\begin{remark}
%Let us mention that, in most applications, we are interested, as 
%shown by \eqref{cond-initiale}, in initial data
%$u_0$ which ``represent'' an initial front, that is the boundary $\partial{K_0} $ of a
%compact subset $K_0$ of $\R^N$; more precisely
%\begin{eqnarray*}
%\{u_0 \geq 0\}=K_0 \quad \text{and} \quad \{u_0 = 0\}=\partial{K_0}.   
%\end{eqnarray*}
%When $\partial K_0$ is smooth, \eqref{borne-inf0} means that
%$|Du_0(x)|\geq \eta >0$ for $x\in \partial K_0.$ 
%For nonsmooth sets $K_0,$ in many cases we can use $u_0(x)= {\rm tanh}
%(\overline{d}_{\partial K_0}(x))$
%($\overline{d}_{\partial K_0}$ denotes the signed distance to the boundary of $K_0$)
%which satisfies {\bf (H3)}.
%\end{remark}

%%%%%%%%%%%%%%
%%%%%%%
\subsection{Applications}\label{sec:appli-unbound}
\subsubsection{Dislocation dynamics equations}

One important example for which Theorem \ref{existence-gene} provides a weak
solution is the dislocation dynamics equation (see \cite{ahlm06}, \cite{bclm07} and the
references therein), which reads
\begin{equation}\label{eq-dislocation} 
\left\{
\begin{array}{ll}
u_t = (c_0 (\cdot ,t) \star \1_{\{u(\cdot,t)\geq 0\}}(x) 
+c_1(x,t))|D u|   & \mbox{in} \; \R^N\times (0,T),\\ 
u(\cdot,0)=u_0 &\mbox{in} \; \R^N, 
\end{array}\right. 
\end{equation} 
where
\begin{eqnarray*} 
c_0 (\cdot ,t) \star \1_{\{u(\cdot,t)\geq 0\}}(x)=\int_{\R^N} c_0(x-y,t)  
\1_{\{u(\cdot,t)\geq 0\}}(y)dy. 
\end{eqnarray*} 
We assume that $c_0$ and $c_1$ satisfy the following assumptions:

\vspace{.2cm}

\noindent{\bf (A)} $(i)$ $c_0\in C^0([0,T]; L^1\left(\R^N \right)),$
$c_1\in C^0(\R^N\times [0,T];\R)$.

\vspace{.2cm}

$(ii)$ $D c_0\in L^\infty ([0,T]; L^1 (\R^N )).$

\vspace{.2cm}

$(iii)$ There exists a constant $M$ such that, for any $x,y\in \R^N$ and $t\in [0,T]$
\begin{eqnarray*}
|c_1(x,t)|\leq M \quad \hbox{and} \quad |c_1(x,t)-c_1(y,t)|\le M|x-y|.
 \end{eqnarray*} 

%%%%%%%%%%%%%%
\begin{theorem} \label{existence-dislo}
Under assumption \noindent{\bf (A)}, Equation \eqref{eq-dislocation} 
has at least a weak solution. Moreover, if, for all 
$(x,t,\chi) \in \R^N \times [0,T]\times L^\infty(\R^N\times [0,T];[0,1]),$
\begin{equation}\label{hyp-cla} 
c_0 (\cdot ,t) \star \chi (\cdot ,t) (x) +c_1(x,t)\geq 0, 
\end{equation} 
and if there exists $\eta >0$ with
\begin{eqnarray}\label{borne-inf0}
|u_0| + |Du_0|\geq \eta \ \ \ {\rm in} \ \R^N \ {\rm in \ the \ viscosity \ sense,}
\end{eqnarray}
then any weak solution is classical.
\end{theorem}
%%%%%%%%%%%%%%

%%%%%%
\begin{proof}
This theorem is proved by Barles, Cardaliaguet, Ley and Monneau 
in \cite[Theorem 1.2]{bclm07}. Another proof can
be done using Theorem \ref{existence-gene}. First we note that
Equation \eqref{eq-dislocation} is a first-order particular case of \eqref{eq-gene}
with, for all $(x,t,p,\chi)\in \R^N\times [0,T] \times  \R^N\times
L^\infty(\R^N\times [0,T];[0,1]),$
\begin{eqnarray*}
H[\chi](x,t,p)= 
[c_0 (\cdot ,t) \star \chi(\cdot ,t) (x) +c_1(x,t)]|p|.  
 \end{eqnarray*} 
Assumption {\bf (H1-$X$)} $(i)$ is given by \cite[Theorem 5.4]{bclm07}
and \cite{ley01} (for the regularity part), 
while assumption {\bf (H1-$X$)} $(ii)$ 
holds thanks to the results of \cite{nunziante92}. Assumption {\bf (H2-$X$)} 
is given by \cite[proof of Theorem 1.2]{bclm07}. It essentially amounts to noticing that if $\displaystyle \chi_n  \mathop{\rightharpoonup} \chi$ 
in $L^\infty\text{-weak-}*$, then by the definition of this convergence
$$
\begin{aligned}
\int_0^t c_0(\cdot,s) \star \chi_n(\cdot,s)(x)\, ds = &\int_0^t \int_{\R^N} c_0(x-y,s) \chi_n(y,s) \, dy ds\\
\to &\int_0^t \int_{\R^N} c_0(x-y,s) \chi(y,s) \, dy ds \\ =&\int_0^t c_0(\cdot,s) \star \chi(\cdot,s)(x) \, ds. 
 \end{aligned}
 $$
Finally, if \eqref{hyp-cla} and \eqref{borne-inf0} hold, the solutions are classical
by \cite[Theorem 1.3]{bclm07}.
\end{proof}

We can also consider the dislocation dynamics equation with an
additionnal mean curvature term,
\begin{equation} \label{ATW:level-set-equation}
u_t=\left[\div \left(\frac{Du}{|Du|}\right) +
c_0(\cdot,t) \star \1_{\{ u(\cdot,t)\geq 0 \}}(x) +
c_1(x,t)\right]|Du|,
\end{equation}
which has been studied by Forcadel and Monteillet \cite{fm07} for instance. 
Theorem \ref{existence-gene} also provides a weak solution 
to \eqref{ATW:level-set-equation}. In \cite{fm07}, however, 
the authors study the problem with the particular tool of 
minimizing movements, which enables them to construct a weak 
solution with $\chi$ of the form $\1_E$, with good regularity 
properties of $t \mapsto E(t)$. This is due to the particular 
structure of \eqref{ATW:level-set-equation}, namely the presence 
of the regularizing mean curvature term. Here we can deal with
more general nonlocal degenerate parabolic equations.

%%%%%%%%%%%%%%%%%%%%%%%%%%%%%%%%%%%%%%%%%%%%%%
\subsubsection{A Fitzhugh-Nagumo type system}

We are also interested in the following system,
\begin{eqnarray} \label{system}
\hspace*{0.9cm}
\left\{
  \begin{array}{ll} 
u_t=\alpha (v)|Du| & {\rm in} \  \R^N \times (0,T),  \\  
v_t-\Delta v = (g^+(v)\1_{\{ u\geq 0 \}}+g^-(v)(1-\1_{\{
  u\geq 0 \}})) & {\rm in} \  \R^N \times (0,T),\\
u(\cdot ,0)=u_0, \ v(\cdot ,0)=v_0  & {\rm in} \  \R^N,
\end{array}    
\right.
\end{eqnarray}  
which is obtained as the asymptotics as $\e \to 0$ of the following
Fitzhugh-Nagumo 
system arising in neural wave propagation or chemical kinetics (see \cite{ss96}):
\begin{eqnarray}\label{system-eps}
\left\lbrace
\begin{array}{l}
u^{\e}_t-\e \Delta u^{\e}=\e^{-1}f(u^{\e},v^{\e}), \\
v^{\e}_t-\Delta v^{\e}=g(u^{\e},v^{\e})
\end{array}
\right.
\end{eqnarray}
in $\R^N \times (0,T)$, where for $(u,v)\in \R^2$,
\begin{eqnarray*}
\left\lbrace
\begin{array}{ll}
f(u,v)= u(1-u)(u-a) -v & (0<a<1), \\
g(u,v)=u-\gamma v & (\gamma >0).
\end{array}
\right.
\end{eqnarray*}
The functions $\alpha$, $g^+$ and $g^-:\R \to \R$ appearing 
in \eqref{system} are associated with $f$ and $g$. This system 
has been studied in particular by Giga, Goto and Ishii \cite{ggi92} and 
Soravia, Souganidis \cite{ss96}. They proved existence of a weak 
solution to \eqref{system}. Here we recover their result as an application of Theorem \ref{existence-gene}.

\vspace{.2cm}

If for $\chi \in L^{\infty}(\R^N \times [0,T];[0,1])$, $v$ denotes the solution of
\begin{eqnarray}\label{heat-eq}
\left\{
\begin{array}{ll}
v_t-\Delta v = g^+(v)\chi+g^-(v)(1-\chi) & {\rm in} \,  \R^N \times (0,T), \\
v(\cdot ,0)=v_0  & {\rm in} \, \R^N,
\end{array} \right.
\end{eqnarray}  
and if $c[\chi](x,t) := \a(v(x,t))$, then Problem \eqref{system} reduces to 
\begin{eqnarray} \label{eq-gene-fn}  
\left\lbrace  
\begin{array}{ll}  
u_t(x,t)=c[\1_{\{ u \geq 0 \}}](x,t)|Du(x,t)| &\text{in} \; \R^N \times (0,T),\\  
u(\cdot,0)=u_0 &\text{in} \; \R^N,
\end{array}
\right.  
\end{eqnarray} 
which is a particular case of \eqref{eq-gene}. Let us now state the 
existence theorem of \cite{ggi92} and \cite{ss96} that we can recover 
from our general existence theorem. We first gather the assumptions 
satisfied by $\a$, $g^-$, $g^+$ and $v_0$:

\vspace{.2cm}

\noindent{\bf (B)} $(i)$ $\alpha$ is Lipschitz continuous on $\R$.

\vspace{.2cm}

$(ii)$ $g^+$ and $g^-$ are Lipschitz continuous 
on $\R^N$, and there exist $\underline{g}$ and $\overline{g}$ in $\R$ such that
$$
\underline{g} \leq g^-(r) \leq g^+(r) \leq \overline{g} 
\quad {\rm for  \ all} \  r \ {\rm in} \ \R.
$$ 

$(iii)$ $v_0$ is bounded and of class $C^1$ with $\|Dv_0\|_{\infty}
<+\infty$.

%%%%%%%%%%%%%%
\begin{theorem} \label{existence-fn}
Under assumption \noindent{\bf (B)}, the problem \eqref{eq-gene-fn}, 
or equivalently the system
\eqref{system}, 
has at least a weak solution. If in addition \eqref{borne-inf0} holds and
$\alpha\geq 0,$ 
then any weak solution is classical.
\end{theorem}
%%%%%%%%%%%%%%

%%%%%%%%%%%%%%
\begin{proof}
The explicit 
resolution of the heat equation \eqref{heat-eq} 
shows that for any $(x,t) \in \R^N \times (0,T)$,
\begin{eqnarray}
&& v(x,t) = \int_{\R^N} G(x-y,t)  \, v_0(y) \, dy \label{form-chal}\\ 
&& \hspace*{1cm}
+ \int_0^t \int_{\R^N} G(x-y,t-s)  \, [g^+(v)\chi +g^-(v)(1-\chi )](y,s) \, dy ds,
\nonumber
\end{eqnarray}
where $G$ is the Green function defined by
\begin{equation} \label{green}
G(y,s)=\frac{1}{(4\pi s)^{N/2}} e^{-\frac{|y|^2}{4s}}.  
\end{equation}

It is then easy to obtain the following lemma:

%%%
\begin{lemma}\label{reg-velocity}  
Assume that $g^-$, $g^+$ and $v_0$ satisfy \noindent{\bf (B)}. 
For $\chi \in L^{\infty}(\R^N \times [0,T];[0,1])$, let $v$ be the
solution of \eqref{heat-eq}. Set $\gamma=\max \{|\underline{g}|,|\overline{g}|\}.$ 
Then there exists a constant $k_N$ depending only on $N$ such that:
\begin{itemize}

\item[$(i)$] $v$ is uniformly bounded: for all $(x,t) \in \R^N \times [0,T],$
$$
|v(x,t)| \leq \|v_0\|_{\infty} + \gamma t.
$$ 

 \item[$(ii)$] $v$ is continuous on $\R^N \times [0,T]$.

 \item[$(iii)$]  For any $t \in [0,T]$, $v(\cdot,t)$ is of 
class $C^1$ in $\R^N$.

 \item[$(iv)$] For all $t \in [0,T]$, for all $x,y \in \R^N$,  
$$|v(x,t)-v(y,t)|\leq (\, \|Dv_0\|_{\infty} +\gamma k_N \,\sqrt{t}) \,|x-y|.$$

 \item[$(v)$]  For all $0\leq s \leq t \leq T,$ for all $x \in \R^N$,  
$$
|v(x,t)-v(x,s)|\leq k_N
(\|Dv_0\|_{\infty}+ \gamma k_N 
\,\sqrt{s}) \,\sqrt{t-s} + \gamma (t-s).
$$

\end{itemize}
\end{lemma}  
%%%
In particular the velocity $c[\chi]$ in \eqref{eq-gene-fn}  
is bounded, continuous on $\R^N \times [0,T]$ and Lipschitz continuous in space, 
uniformly with respect to $\chi$. From general results on existence and comparison 
for classical viscosity solutions of the eikonal equation 
with Lipschitz continuous initial datum (see for instance
\cite[Theorem 2.1]{bclm07}), we obtain that {\bf (H1-$X$)} is satisfied.

\vspace{.2cm}

Let us check \noindent{\bf (H2-$X$)} $(iii)$ ($(i)$ and $(ii)$ are straightforward): 
we claim that, if $\displaystyle \chi_n  \mathop{\rightharpoonup} \chi$ 
in $L^\infty\text{-weak-}*$, then $$\int_0^t c[\chi_n](x,s) \, ds 
\to \int_0^t c[\chi](x,s) \, ds $$ locally uniformly in $[0,T].$ 
Indeed, let $v_n$ (resp. $v$) be the solution of \eqref{heat-eq} 
with $\chi_n$ (resp. $\chi$) in the right-hand side. The estimates 
$(iv)$ and $(v)$ of Lemma \ref{reg-velocity} on the heat equation 
imply that we can extract by a diagonal argument a subsequence, still denoted $(v_n)$, which converges uniformly to some $w$ 
in $\bar{B}(0,R) \times [0,T]$ for any $R>0$. We know that for any $(x,t) 
\in \R^N \times (0,T)$,
\begin{eqnarray*}
&& v_n(x,t)=\int_{\R^N} G(x-y,t)  \, v_0(y) \, dy \\ 
&& \hspace*{1cm}+ \int_0^t \int_{\R^N} G(x-y,t-s)  \, [g^+(v_n)\chi_n+g^-(v_n)(1-\chi_n)](y,s) \, dy ds
\end{eqnarray*}
where $G$ is the Green function 
defined by \eqref{green}. As $n$ goes to infinity, we obtain
\begin{eqnarray*}
&& w(x,t)=\int_{\R^N} G(x-y,t)  \, v_0(y) \, dy  \\ 
&& \hspace*{1cm}+ \int_0^t \int_{\R^N} G(x-y,t-s)  \, [g^+(w)\chi+g^-(w)(1-\chi)](y,s) \, dy ds.
\end{eqnarray*}
Indeed
\begin{eqnarray*}
&&\int_0^t \int_{\R^N} G(x-y,t-s)  \, [g^+(v_n)\chi_n+g^-(v_n)(1-\chi_n)](y,s) \, dy ds\\
&& -\int_0^t \int_{\R^N} G(x-y,t-s)  \, [g^+(w)\chi+g^-(w)(1-\chi)](y,s)  \, dy ds\\
&=&\int_0^t \int_{\R^N} G(x-y,t-s) \, [(g^+(w)-g^-(w))(\chi_n-\chi)](y,s) \,  dy ds \\   
&& +\int_0^t \int_{\R^N} G(x-y,t-s) [\chi_n \, (g^+(v_n)-g^+(w)) + \\
&& \hspace*{3cm}\chi_n \, (g^-(w)-g^-(v_n))
+(g^-(v_n)-g^-(w))](y,s) \,  dy ds.
\end{eqnarray*}
The term
$$
\int_0^t \int_{\R^N} G(x-y,t-s) \, [(g^+(w)-g^-(w))(\chi_n-\chi)](y,s) \, dyds
$$
converges to 0 since $\displaystyle \chi_n  \mathop{\rightharpoonup} \chi$ 
in $L^\infty\text{-weak-}*$ and 
$$
|G(x-y,t-s)  \, (g^+(w)-g^-(w))(y,s)| \leq (\overline{g}-\underline{g}) \, G(x-y,t-s),
$$
which is an integrable function of $(y,s)$. The rest of the terms
converges to 0 
by dominated convergence since $v_n \to w$ pointwise in $\R^N \times [0,T]$ and
\begin{eqnarray*}
&& | \chi_n \, (g^+(v_n)-g^+(w)) + \chi_n \,
(g^-(w)-g^-(v_n))+(g^-(v_n)-g^-(w))|  \\
&\leq &  3M |v_n-w| \\
 &\leq & 6M C,
\end{eqnarray*}
where $M$ is a Lipschitz constant for $g^+$ and $g^-$, and $C$ 
is a uniform bound for $v_n$ and $w$ given by Lemma \ref{reg-velocity} $(i)$.

This shows that $w$ is the solution of \eqref{heat-eq}, so that $w=v$. 
In particular $(v_n)$ converges locally uniformly to $v$. We conclude that 
$$\int_0^t c[\chi_n](x,s)\, ds \to \int_0^t c[\chi](x,s)\, ds $$ 
locally uniformly for $t\in [0,T]$ thanks to the Lipschitz 
continuity of $\alpha$. This proves the claim that 
{\bf (H2-$X$)} holds, and we obtain existence of weak solutions 
to \eqref{eq-gene-fn} according to Theorem \ref{existence-gene}.

If $\alpha \geq 0$ and \eqref{borne-inf0} holds, then the fattening 
phenomenon for \eqref{eq-gene2} does not happen
(see \cite{bss93, ley01}) so that, if $u$ is any weak 
solution of \eqref{eq-gene2}, 
then for almost all $t \in [0,T]$ and almost everywhere in $\R^N$,  
$$  
\1_{\{u(\cdot,t)>0\}}=\1_{\{u(\cdot,t)\geq 0\}},  
$$  
which means that $u$ is a classical solution of \eqref{eq-gene-fn}. 
This completes the proof. 
\end{proof}

%%%%%%%%%%%%%%%%%%%%%%%%%%%%%%%%%%%%%%%%%%%%%%%%%%%%%%%%%%%%%%%%%%%%%%
%%%%%%%%%%%%%%%%%%%%%%%%%%%%%%%%%%%%%%%%%%%%%%%%%%%%%%%%%%%%%%%%%%%%%%
%%%%%%%%%%%%%%%%%%%%%%%%%%%%%%%%%%%%%%%%%%%%%%%%%%%%%%%%%%%%%%%%%%%%%%
\section{Existence of weak solution to \eqref{eq-gene} (bounded case)}

It may happen that our Hamiltonian $ H[\1_{\{ u\geq 0 \}}]$ is only defined when 
the set $\{ u\geq 0 \}$ remains bounded: this is typically the case 
when a volume term is involved. For such cases, the existence of weak 
solutions may remain true, due to a particular framework.

%%%%%%%%%%%%%%%%%%%%%%%%%%%%%%%%%%%%%%%%%%%%%%%%%%%%%%%%%%%%%%%%%%%%%%
\subsection{The existence theorem}

We use the following assumption:

\vspace{.2cm}

\noindent{\bf (H4)} There exists a bounded function $\overline{v}:\R^N\times [0,T]\to \R$ and $R_0 >0$ such that
\begin{itemize}

\item[$(i)$] $ \overline{v}(x,t) <0 \quad \hbox{if   }|x| \geq R_0\; ,\text{for any} \; t\in [0,T],$

\vspace{.2cm}

\item[$(ii)$] $ \overline{v}(x,0) \geq u_0 (x) \quad \hbox{in   } \R^N$,

\vspace{.2cm}

\item[$(iii)$] $\overline{v}$ is a supersolution of \eqref{eq-gene2} for all $\chi \in X$, where 
$$
X = \{\chi \in L^{\infty}(\R^N \times [0,T];[0,1])\ : \ 
\chi = 0 \; \hbox{a.e in  }\{\overline{v} <0\}\}\; .
$$
\end{itemize}
Assumption {\bf (H4)} is some kind of compatibility condition between 
the equation and the initial condition: of course, it implies that 
$u_0(x) <0$ if $|x|\geq R_0$ and, more or less, that the equation preserves 
this property (this is the meaning of $\overline{v}$).

%One of the main differences with the previous case is the fact 
%that the Hamiltonians $H[\chi](x,t,p,A)$ can only depend on the 
%past, that is, on $\chi(\cdot,s)$ for $0\leq s\leq t$, as is 
%implicitely required by assumption {\bf (H1')} $(iii)$. This 
%can also be written as follows: for any $t\in [0,T]$, 
%if $\chi_1, \chi_2 \in L^{\infty}(\R^N \times [0,T];[0,1])$ 
%have compact support and satisfy $\chi_1(\cdot,s)=\chi_2(\cdot,s)$ 
%for almost all $s\in [0,t]$, then for all 
%$(x,p,A)\in \R^N \times  \R^N \setminus \{0\} \times \mathcal{S}_N,$
%$$
%H[\chi_1](x,t,p,A)=H[\chi_2](x,t,p,A).
%$$ This restriction comes from the fact that these problems 
%enjoy less compactness than in the unbounded case.*********

Under this assumption, we obtain the following existence result:
%%%%%%%%%%%%%%%%%%%%%
\begin{theorem}  \label{existence-gene2}
Assume {\bf (H4)} and that {\bf (H1-$X$), (H2-$X$)} and {\bf
(H3-$X$)} hold with $X$ given in {\bf (H4)}.
Then there exists at least a weak solution to \eqref{eq-gene}. 
\end{theorem} 
%%%%%%%%%%%%%%%%%%%%%

\begin{proof} The proof follows the arguments of the proof of 
Theorem~\ref{existence-gene}: essentially, the only change is the choice of $X$.

From {\bf (H4)}, the set-valued mapping 
\begin{eqnarray*}
\begin{array}{cccl}  
\xi : & X & \rightrightarrows &  X\\
      &  \chi                            &\mapsto            &  \big\{
      \chi'\; : \; \1_{\{ u(\cdot,t) > 0 \}}  \leq \chi'(\cdot,t) \leq
      \1_{\{ u(\cdot,t) \geq 0 \}} \; \text{for
        almost all} \;  t \in [0,T] ,\\ 
      &                                  &                   & 
      \quad \text{where} \; u \; \text{is the $L^1$-viscosity solution
        of} 
\; \eqref{eq-gene2} \big\},  
\end{array}
\end{eqnarray*}  
is well-defined : indeed, for any $\chi \in X$, $\overline{v}$ is a supersolution of 
the $H[\chi]$-equation and we have $ u_0(x) \leq \overline{v}(x,0)$ in $\R^N$. 
Therefore, by {\bf (H1-$X$)} (comparison),
$$  u(x,t) \leq \overline{v}(x,t) 
\quad \hbox{in   }\R^N\times [0,T]\; .$$
In particular, $u(x,t) < 0$ if $\overline{v}(x,t) <0$ 
and clearly any $\chi'$ in $\xi(\chi)$ is in $X$.

We conclude exactly as in the proof of Theorem~\ref{existence-gene}.
\end{proof}

%%%%%%%%%%%%%%%%%%%%%%%%%%%%%%%%%%%%%%%%%%%%%%%%%%%%%%%
\subsection{Applications}

The typical cases we have in mind are geometrical equations;
for instance, 
$$ u_t=\left[\div \left(\frac{Du}{|Du|}\right) 
+\beta\left(\Lc^N(\{u(\cdot,t)\geq 0\})\right)\right]|Du|\; ,
$$
where $\beta : \R \to \R$ is a continuous function
and $\Lc^N$ denotes the Lebesgue measure in $\R^N.$

In order to test condition {\bf (H4)}, it is natural to consider 
radially symmetric supersolutions and typically, we look for 
supersolutions of the following form:
$$ \psi(x,t) := R(t) - |x|\; ,$$
where $R(\cdot)$ is a $C^1$-function of $t$. We point out two 
key arguments to justify this choice: first $\psi$ is concave 
in $x$, and checking the viscosity supersolution property is 
equivalent to checking it at points where $\psi$ is smooth 
(because of the form of $\psi$). Next if $\psi$ is a 
supersolution of the above pde, one can use (if necessary) 
a change of function $\psi \to \varphi(\psi)$, with 
$\varphi'>0$ to ensure that $\varphi(\psi) \geq u_0$ in $\R^N$, and such that $\varphi$ is bounded, in order to use the comparison principle. 

Plugging $\psi$ in the equation, we obtain that $\psi$ is a supersolution if
$$
R'(t) \geq -\frac{(N-1)}{|x|} + \beta (\omega_N R^N(t))\;,
$$
where $\omega_N={\mathcal L}^N(B(0,1)).$
The curvature term $(N-1)/|x|$ is not going 
to play any major role here since we are concerned with large $R$'s and the equation should hold on the $0$-level set of $\psi$, 
i.e. for $|x| = R$.

Therefore let us consider the ordinary differential equation (ode)
$$R'(t) = \beta (\omega_N R^N(t))\quad \hbox{with  }R(0) = R_0\; .$$
A natural condition for this ode to have solutions which do not 
blow up in finite time is the sublinearity in $R$ of the right-hand 
side. This leads to the following conditions on $\beta$:
$$ \beta (t) \leq L_1 + L_2 t^{1/N} \quad \hbox{for any  }t>0\; ,$$
for some constants $L_1,L_2>0$. Under this condition, we easily build a function $\overline{v}$ satisfying ${\bf (H4)}$. However if this condition is not satisfied, Theorem~\ref{existence-gene2} 
provides only the small time existence of solutions.

We complete this example by recalling that, for all 
$\chi \in L^{\infty}(\R^N \times [0,T];[0,1])$ with bounded support,  
we have a comparison result for the equation
$$ u_t=\left[\div \left(\frac{Du}{|Du|}\right) 
+\beta\left(\int_{\R^N}\chi(x,t)dx)\right)\right]|Du|\quad \hbox{in  }\R^N \times [0,T]
$$
(See Nunziante \cite{nunziante90} and Bourgoing \cite{bourgoing04a,bourgoing04b}).

%
%\hspace*{0.5cm} $(iii)$ There exists a constant $L\geq 0$ 
%such that for all $R>0$, for almost all $t \in [0,T]$, 
%for all $\chi\in L^{\infty}(\R^N \times [0,T];[0,1])$ 
%with ${\rm supp}(\chi(\cdot,s)) \subset \bar{B}(0,R)$ if 
%$s\leq t$, for all $x \in \bar{B}(0,R)$, for all $p \in \R^N$ with $|p|=1$,
%$$
%H[\chi](x,t,p,-\frac{1}{R} I_N) \leq L(1+R).
%$$

%If {\bf (H1')} holds, then {\bf (H5)} holds.

%
%Using Theorem {existence-gene},
%we can easily show existence of weak 
%solutions for equations of the form 

%%%%%%%

%%%%%%%%%%%%%%%%%%%%%%%%%%%%%%%%%%%%%%%%%%%%%
%%% BIBLIO
%%%%%%%%%%%%%%%%%%%%%%%%%%%%%%%%%%%%%%%%%%%%%
%%%% AVEC BIBTEX
%\bibliographystyle{plain} 
%\bibliography{biblio} 
%\end{document}
%%%%%%%%%%%%%%%%%%%%%%%%%%%%%%

\end{document}